\documentclass{article}
\usepackage{amsmath}
\usepackage{amssymb}
\usepackage{amsfonts}

\newtheorem{theorem}{Theorem}

\newtheorem{corollary}[theorem]{Corollary}

\newtheorem{lemma}[theorem]{Lemma}

\newtheorem{remark}[theorem]{Remark}

\newenvironment{proof}[1][Proof]{\noindent\textbf{#1.} }{\ \rule{0.5em}{0.5em}}
\input{tcilatex}

\begin{document}

\title{Extremal inscribed and circumscribed complex~ellipsoids}
\author{Jorge L. Arocha, Javier Bracho and Luis Montejano}
\maketitle

\begin{abstract}
We prove that if a convex set in $\mathbb{C}^{n}$ contains two
inscribed complex ellipsoid of maximal volume then one is a translate of the
other. On the other hand, the circumscribed complex elipsoid of minimal
volume is unique. As application we prove the complex analoge of Brunn's
characterization of ellipsods.
\end{abstract}

\tableofcontents

\section{Introduction}

Let $A$ a non-flat compact subset of the euclidean space $\mathbb{R}^{n}$.
Denote by $\widehat{A}$ the convex closure of $A$. A solid ellipsoid $%
\mathfrak{E}$ is called circumscribed if $A\subset \mathfrak{E}$. On the
other hand, it is called inscribed if $\mathfrak{E}\subset \widehat{A}$. If
among all circumscribed ellipsoids $\mathfrak{E}$ has the minimal volume we
say that $\mathfrak{E}$ is a minimal circumscribed ellipsoid (MiCE). On the
other hand, if among all inscribed ellipsoids $\mathfrak{E}$ has the maximal
volume we say that $\mathfrak{E}$ is a maximal inscribed ellipsoid (MaIE).

The existence of these ellipsoids is proved using standard arguments. Due to
the properties of $A$ there is a finite sphere big enough that contains $A$
and a non cero sphere which is contained in $\widehat{A}$. All ellipsoids
are easily parametrized by a matrix and a vector and among them we can
consider only those that contain the small sphere and are contained in the
big sphere. This is a compact in the parameter space. Moreover the volume
function is continuous and the existence follows.

The real interesting thing about MiCE and MaIE is that they are unique. The
first proofs of this fact in its full generality seems to appear
independently in \cite{Danzer} and \cite{Zaguskin}. MiCE and MaIE are known
as L\"{o}wner--John Ellipsoids and have many applications in several areas
of mathematics (see \cite{Martin} and the references there).

In this paper we deal with similar questions in the space $\mathbb{C}^{n}$.
In this space the ellipsoids are the unit balls for the norms defined by
inner products. The only result previously available seems to be in the
paper by Gromov \cite{Gromoff} where he proved the uniqueness of the
circumscribed ellipsoid minimal among those centered at the origin, when $A$
is an unit ball of a Banach space over $\mathbb{C}$ (\cite{Gromoff} Lemma 1).

We prove in section 3 that If a convex set (convexity inherited from $%
\mathbb{R}^{2n}$) contains two complex MaIE then one is a translate of the
other. In section 4 we prove that the complex MiCE is unique. In section 5
we give some applications. The most robust result there, is the analog of
the Brunn's Theorem for complex ellipsoids.

Whenever possible, we use a unifying approach for the real and complex cases.

\section{Preliminaries}

We denote by $\mathbb{K}$ a field. Moreover, here $\mathbb{K}\ $is $\mathbb{R%
}$ or $\mathbb{C}$. The elements of $\mathbb{K}^{n}$ will be denoted in
boldface. Often, the geometric terminology is used. Points are the vectors
of $\mathbb{K}^{n},$ affine subspaces of dimension $1$ of $\mathbb{K}^{n}$
are called lines; affine subspaces of codimension $1$ of $\mathbb{K}^{n}$
are called hyperplanes. We emphasize that, unless the contrary is explicitly
stated, all objects are general i.e. for example, a line is not specifically
a real or complex line; it is a line in $\mathbb{K}^{n}.$

For $\mathbf{x}=\left( x_{1},...,x_{n}\right) \in \mathbb{K}^{n}$ denote by $%
\overline{\mathbf{x}}=\left( \bar{x}_{1},...,\bar{x}_{n}\right) $, where the
bar above is the complex conjugate. Also, we denote by $\mathbf{x}\odot 
\mathbf{y}$ the Hadamar product of $\mathbf{x}$ and $\mathbf{y}$. That is,
the coordinatewise product (see for example \cite{Horn} Chapter 5).

An scalar product in $\mathbb{K}^{n}$ is a sesquilinear, Hermitian, definite
positive functional $\mathbb{K}^{n}\times \mathbb{K}^{n}\rightarrow \mathbb{K%
}$ denoted by $\left\langle \mathbf{x}\cdot \mathbf{y}\right\rangle $ for
any $\mathbf{x}$ and $\mathbf{y}$ in $\mathbb{K}^{n}$. For each scalar
product there is a matrix $A$ (Hermitian, definite positive) such that $%
\left\langle \mathbf{x}\cdot \mathbf{y}\right\rangle =\mathbf{x}^{T}A%
\overline{\mathbf{y}}$. In the case that $A$ is the identity matrix, the
inner product is the usual Hermite's product in $\mathbb{K}^{n}$. It is very
well known that the eigenvalues of $A$ must be real positive numbers and
also its determinant. It is also known that $A$ is diagonalizable by a
unitary transformation.

An ellipsoid (centered at the origin) is a set%
\begin{equation}
\left\{ \mathbf{x}\in \mathbb{K}^{n}\mid \mathbf{x}^{T}A\overline{\mathbf{x}}%
\leq 1\right\} .  \label{Elipsoid}
\end{equation}%
For the case that $A\ $is the identity matrix this ellipsoid is the unit
ball $\mathfrak{B}\left( \mathbb{K}^{n}\right) $ and its boundary is the
unit sphere $\mathfrak{S}\left( \mathbb{K}^{n}\right) $. The set of unitary
transformations is the subgroup of $GL\left( \mathbb{K}^{n}\right) $ that
preserves the unit sphere. The modulus of an scalar $\lambda \in \mathbb{K}$
will be denoted by $\left\vert \lambda \right\vert $. The set $\mathfrak{S}%
\left( \mathbb{K}^{1}\right) $ is the multiplicative group of scalars with
modulus $1$ in $\mathbb{K}$. We will denote by $\left\Vert \cdot \right\Vert 
$ the usual norm in $\mathbb{K}^{n}$, i.e. the norm defined by the Hermite's
product.

Denoting $B\overset{\text{def}}{=}\sqrt{A^{T}}$,$\ $we have $A^{T}=BB$ and
therefore $A=B^{T}B^{T}=B^{T}\overline{B}$. From this we obtain

\begin{equation*}
\mathbf{x}^{T}A\overline{\mathbf{x}}=\mathbf{x}^{T}B^{T}\overline{B}%
\overline{\mathbf{x}}=\left( B\mathbf{x}\right) ^{T}\overline{\left( B%
\mathbf{x}\right) }=\left\Vert B\mathbf{x}\right\Vert ^{2}
\end{equation*}%
and therefore, the ellipsoids in $\mathbb{K}^{n}$ can be written in the form%
\begin{equation*}
\left\{ B^{-1}\mathbf{u\mid u}\in \mathfrak{B}\left( \mathbb{K}^{n}\right)
\right\} 
\end{equation*}

If we use a unitary transformation to bring $A$ to the diagonal form then
this can be rewritten as%
\begin{equation}
\mathfrak{El}\left( \mathbf{\lambda }\right) \overset{\text{def}}{=}\left\{ 
\mathbf{\lambda }\odot \mathbf{u\mid u}\in \mathfrak{B}\left( \mathbb{K}%
^{n}\right) \right\}   \label{Forma parametrica}
\end{equation}%
where $\mathbf{\lambda }$ is a vector in $\mathbb{R}_{+}^{n}$. Since $A^{%
{\frac12}%
}=B$, the two forms \ref{Elipsoid} and \ref{Forma parametrica} are related
by the fact that $\mathbf{\lambda }$ is the diagonal of $A^{-%
{\frac12}%
}$. The map $\mathbf{x\mapsto \lambda }\odot \mathbf{x}=A^{-%
{\frac12}%
}\mathbf{x}$ is an invertible linear map in $GL\left( \mathbb{K}^{n}\right) $
which maps the unit sphere $\mathfrak{B}\left( \mathbb{K}^{n}\right) $ into
the ellipsoid $\mathfrak{El}\left( \mathbf{\lambda }\right) .$ Therefore 
\begin{equation}
\func{Vol}\mathfrak{El}\left( \mathbf{\lambda }\right) =\det A^{-%
{\frac12}%
}\func{Vol}\mathfrak{B}\left( \mathbb{K}^{n}\right) .  \label{Volumen-det}
\end{equation}

The results that follow do not depend on the measure chosen to define the
volume. We just need the validity of the equation \ref{Volumen-det}.

We will denote $\det \mathbf{\lambda }$ the product of coordinates of $%
\mathbf{\lambda }$. Of course\ we have $\det \mathbf{\lambda }=\det A^{-%
{\frac12}%
}$ and%
\begin{equation*}
\func{Vol}\mathfrak{El}\left( \mathbf{\lambda }\right) =\det \mathbf{\lambda 
}\func{Vol}\mathfrak{B}\left( \mathbb{K}^{n}\right) .
\end{equation*}

The translates of ellipsoids centered at the origin are also ellipsoids.
Translations do not change volume.

\section{Maximal inscribed ellipsoids.}

\begin{theorem}
\label{MaIE}Let $A$ a non-flat compact in $\mathbb{K}^{n}$. Let $\mathfrak{E}%
_{1}$ and $\mathfrak{E}_{2}$ be two MaIE contained in $\widehat{A}$. Then,
there is a vector $\mathbf{c}\in \mathbb{K}^{n}$ such that $\mathfrak{E}_{2}=%
\mathfrak{E}_{1}+\mathbf{c}$. If $\mathbb{K}=\mathbb{R}$ then $\mathfrak{E}%
_{1}=\mathfrak{E}_{2}$.
\end{theorem}

\begin{proof}
Using a suitable affine transformation we can suppose that $\mathfrak{E}_{1}$
is the unit ball $\mathfrak{El}\left( \mathbf{1}\right) .$ Suppose that the
center of $\mathfrak{E}_{2}$ is the vector $\mathbf{c}\in \mathbb{K}^{n}$.
We can use a unitary transformation to diagonalize the matrix of $\mathfrak{E%
}_{2}-\mathbf{c}$. And therefore $\mathfrak{E}_{2}=\mathfrak{El}\left( 
\mathbf{\lambda }\right) +\mathbf{c}$ for some $\mathbf{\lambda }\in \mathbb{%
R}_{+}^{n}$.

Let us prove first that the ellipsoid $\mathfrak{E}_{3}=\mathfrak{El}\left( 
{\frac12}%
\left( \mathbf{\lambda }+\mathbf{1}\right) \right) +%
{\frac12}%
\mathbf{c}$ is contained in the convex closure of $\mathfrak{E}_{1}\cup 
\mathfrak{E}_{2}$, i.e. 
\begin{equation*}
\mathfrak{El}\left( \frac{\mathbf{\lambda }+\mathbf{1}}{2}\right) +\frac{%
\mathbf{c}}{2}\subset \widehat{\mathfrak{El}\left( \mathbf{1}\right) \cup
\left( \mathfrak{El}\left( \mathbf{\lambda }\right) +\mathbf{c}\right) }.
\end{equation*}%
Indeed if $\mathbf{x}\in \mathfrak{E}_{3}$, then for some $\mathbf{u}\in 
\mathfrak{B}\left( \mathbb{K}^{n}\right) =\mathfrak{El}\left( \mathbf{1}%
\right) =\mathfrak{E}_{1}$ we have 
\begin{equation*}
\mathbf{x=}\frac{\mathbf{\lambda }+\mathbf{1}}{2}\odot \mathbf{u+}\frac{%
\mathbf{c}}{2}
\end{equation*}%
The vector $\mathbf{y=\lambda }\odot \mathbf{u+c}$ is in $\mathfrak{E}_{2}$
and%
\begin{equation*}
\frac{\mathbf{y}+\mathbf{u}}{2}=\frac{\mathbf{\lambda }\odot \mathbf{u+c+u}}{%
2}=\mathbf{x}
\end{equation*}%
This means that $\mathbf{x}$ is the middle point of the segment joining the
points $\mathbf{y}$ and $\mathbf{u}$. This proves that $\mathfrak{E}%
_{3}\subset \widehat{\mathfrak{E}_{1}\cup \mathfrak{E}_{2}}$. Since $%
\mathfrak{E}_{1}\cup \mathfrak{E}_{2}\subset \widehat{A}$ we have $\mathfrak{%
E}_{3}\subset \widehat{\mathfrak{E}_{1}\cup \mathfrak{E}_{2}}\subset 
\widehat{A}$. i.e. $\mathfrak{E}_{3}$ is also contained in $\widehat{A}$.

We know that $\func{Vol}\left( \mathfrak{E}_{1}\right) =\func{Vol}\left( 
\mathfrak{E}_{2}\right) =\det \mathbf{\lambda }\func{Vol}\left( \mathfrak{E}%
_{1}\right) $ therefore $\det \mathbf{\lambda }=1.$ Moreover $\func{Vol}%
\left( \mathfrak{E}_{3}\right) =\det \left( 
{\frac12}%
\left( \mathbf{\lambda }+\mathbf{1}\right) \right) \func{Vol}\left( 
\mathfrak{E}_{1}\right) .$ If $\mathbf{\lambda }\neq \mathbf{1}$ then by
lemma \ref{volumen} on page \pageref{volumen} $\det \left( 
{\frac12}%
\left( \mathbf{\lambda }+\mathbf{1}\right) \right) >1$ and therefore $\func{%
Vol}\left( \mathfrak{E}_{3}\right) >\func{Vol}\left( \mathfrak{E}_{1}\right)
.$ This contradicts that $\mathfrak{E}_{1}$ is a MaIE. So, we conclude that $%
\mathbf{\lambda }=\mathbf{1}$ and therefore $\mathfrak{E}_{2}=\mathfrak{E}%
_{1}+\mathbf{c}$. This concludes the first part of the theorem.

For the second, $\mathbb{K}=\mathbb{R}$ and we can suppose that $\mathfrak{E}%
_{1}$ and $\mathfrak{E}_{2}$ are two unit balls whose centers are at
distance $2\alpha .$ Let $\mathbf{e}_{1}=\left( 1,0,....,0\right) $ be the
first basis vector. After a suitable unitary affine transformation we can
suppose that $\mathfrak{E}_{1}$ has its center in $\alpha \mathbf{e}_{1}$
and $\mathfrak{E}_{2}$ has its center in $-\alpha \mathbf{e}_{1}$. Now, we
shall prove that $\mathfrak{E}_{3}=\mathfrak{El}\left( \alpha \mathbf{e}_{1}+%
\mathbf{1}\right) $ is contained in the convex closure of $\mathfrak{E}%
_{1}\cup \mathfrak{E}_{2}$ i.e.%
\begin{equation*}
\mathfrak{El}\left( \alpha \mathbf{e}_{1}+\mathbf{1}\right) \subset \widehat{%
\left( \mathfrak{El}\left( \mathbf{1}\right) +\alpha \mathbf{e}_{1}\right)
\cup \left( \mathfrak{El}\left( \mathbf{1}\right) -\alpha \mathbf{e}%
_{1}\right) }.
\end{equation*}%
Indeed if $\mathbf{x}\in \mathfrak{El}\left( \alpha \mathbf{e}_{1}+\mathbf{1}%
\right) ,$ then there exists $\mathbf{u}=\left( u_{1},...,u_{n}\right) \in 
\mathfrak{B}\left( \mathbb{R}^{n}\right) $ such that $\mathbf{x}=\left(
\alpha \mathbf{e}_{1}+\mathbf{1}\right) \odot \mathbf{u}$.

Let $\mathbf{y}=\mathbf{u}+\alpha \mathbf{e}_{1}\in \mathfrak{E}_{1}$ and $%
\mathbf{z}=\mathbf{u}-\alpha \mathbf{e}_{1}$. We shall see that $\mathbf{x}$
is in the segment joining the points $\mathbf{y}$ and $\mathbf{z}$, i.e. we
have to find $t\in \left[ 0,1\right] \subset \mathbb{R}$ such that $\mathbf{x%
}=t\mathbf{y}+\left( 1-t\right) \mathbf{z.}$ In any coordinate but the first
any value of $t$ is appropriate. In the first coordinate we have the
equation $\left( \alpha +1\right) u_{1}=\left( u_{1}+\alpha \right) t+\left(
1-t\right) \left( u_{1}-\alpha \right) $. If $\mathfrak{E}_{1}\neq \mathfrak{%
E}_{2}$ then $\alpha \neq 0$ and the solution of the equation is $t=%
{\frac12}%
\left( u_{1}+1\right) $. The coordinates of vectors in the unit ball can be
any real number in the interval $\left[ -1,1\right] $ hence $t$\ is in the
interval $\left[ 0,1\right] $. This proves that $\mathfrak{E}_{3}\subset 
\widehat{\mathfrak{E}_{1}\cup \mathfrak{E}_{2}}$ and therefore $\mathfrak{E}%
_{3}\subset \widehat{A}$. We have 
\begin{equation*}
\frac{\func{Vol}\left( \mathfrak{E}_{3}\right) }{\func{Vol}\left( \mathfrak{E%
}_{1}\right) }=\det \left( \alpha \mathbf{e}_{1}+\mathbf{1}\right) =\alpha
+1.
\end{equation*}%
If $\mathfrak{E}_{1}\neq \mathfrak{E}_{2},$ then $\alpha >0$ and therefore $%
\func{Vol}\left( \mathfrak{E}_{3}\right) >\func{Vol}\left( \mathfrak{E}%
_{1}\right) ,$ which contradicts the maximality of $\mathfrak{E}_{1}$. This
proves that $\mathfrak{E}_{1}=\mathfrak{E}_{2}$.
\end{proof}

\begin{remark}
The proof of the second part of the theorem does not work for complex
ellipsoids because we are using essentially that $u_{1}$ is a real number.
Moreover, the second part of this theorem is not true for complex
ellipsoids. In $\mathbb{C}^{1}$ the complex ellipsoids are just the disks. A
rectangle of sides 2 and 4 contains many unit disks which are of the maximal
volume. It is easy to generalize this counterexample to all dimensions.
\end{remark}

\subsection{Symmetry}

A subset of $\mathbb{K}^{n}$ is symmetric (centrally) if it is preserved by
the multiplicative group of all scalars of modulus $1$ in $\mathbb{K}$. It
is clear that the origin is the center of any symmetric set. A translate of
a symmetric set is also called symmetric. It is easy to see that an affine
image of a symmetric set is also symmetric. All ellipsoids are symmetric as
they are affine images of the unit ball.

In $\mathbb{C}^{n},$ additionally to the complex ellipsoids, there are real
ellipsoids which are inherited from $\mathbb{R}^{2n}$. Observe that all
complex ellipsoids are real but not the other way around. The difference
between real and complex ellipsoids is clearly explained by the following
result from \cite{BM}.

\begin{theorem}
\label{BM}Any symmetric ellipsoid in $\mathbb{C}^{n}$ is complex.
\end{theorem}

Now we shall see that symmetry guarantees uniqueness of the MaIE.

\begin{theorem}
If $A$ is a compact convex symmetric set in $\mathbb{C}^{n},$ then its
complex MaIE is unique.
\end{theorem}

\begin{proof}
We can suppose that the center of $A$ is the origin. Let $\mathfrak{E}$ be
the unique real MaIE in $A$ (Theorem \ref{MaIE}). Let $\zeta $ be a scalar
of modulus $1$. Since $\zeta \mathfrak{E}\subset \zeta A=A\ $and $\func{Vol}%
\left( \zeta \mathfrak{E}\right) =\func{Vol}\left( \mathfrak{E}\right) $
hence the uniqueness of $\mathfrak{E}$ implies that $\mathfrak{E}$ is
symmetric. Using Theorem \ref{BM} we get that $\mathfrak{E}$ is complex.
There is not another complex MaIE because any complex ellipsoid is real.
\end{proof}

\section{Minimal circumscribed ellipsoids}

\begin{theorem}
\label{MiCE}Let $A$ be a non-flat compact set in $\mathbb{K}^{n}$. Let $%
\mathfrak{E}_{1}$ and $\mathfrak{E}_{2}$ be two MiCE containing $A$. Then, $%
\mathfrak{E}_{1}=\mathfrak{E}_{2}$.
\end{theorem}

\begin{proof}
Again, using an affine transformation we make $\mathfrak{E}_{2}$ the unit
sphere and diagonalize the matrix of $\mathfrak{E}_{1}$ using an unitary
transformation. After that, we can translate so that the center of $%
\mathfrak{E}_{1}$ is the origin and the center of $\mathfrak{E}_{2}$ is some
vector $\mathbf{c}\in \mathbb{K}^{n}$.

Then,%
\begin{equation*}
\begin{array}{l}
\mathfrak{E}_{1}=\left\{ x\in \mathbb{K}^{n}\mid \sum \lambda _{i}\left\vert
x_{i}\right\vert ^{2}\leq 1\right\} =\mathfrak{El}\left( \mathbf{\beta }%
\right)  \\ 
\mathfrak{E}_{2}=\left\{ x\in \mathbb{K}^{n}\mid \sum \left\vert
x_{i}-c_{i}\right\vert ^{2}\leq 1\right\} =\mathfrak{El}\left( \mathbf{1}%
\right) +\mathbf{c}%
\end{array}%
\end{equation*}%
where $\mathbf{\beta }=\mathbf{\lambda }^{-%
{\frac12}%
}.$ Therefore%
\begin{equation*}
A\subset \mathfrak{E}_{1}\cap \mathfrak{E}_{2}\subset \mathfrak{E}_{3}%
\overset{\text{def}}{=}\left\{ x\in \mathbb{K}^{n}\mid \sum \lambda
_{i}\left\vert x_{i}\right\vert ^{2}+\left\vert x_{i}-c_{i}\right\vert
^{2}\leq 2\right\} .
\end{equation*}%
Using Lemma \ref{Lema para el afin} the inequality is transformed into%
\begin{equation}
\sum \left( \lambda _{i}+1\right) \left\vert \left( x_{i}-\frac{c_{i}}{%
\left( \lambda _{i}+1\right) }\right) \right\vert ^{2}\leq 2-\sum \left( 
\frac{\lambda _{i}}{\lambda _{i}+1}\right) \left\vert c_{i}\right\vert
^{2}\leq 2.  \label{desigualdad}
\end{equation}

Let%
\begin{equation*}
\mathfrak{E}_{4}\overset{\text{def}}{=}\left\{ x\in \mathbb{K}^{n}\mid \sum
\left( \lambda _{i}+1\right) \left\vert \left( x_{i}-\frac{c_{i}}{\left(
\lambda _{i}+1\right) }\right) \right\vert ^{2}\leq 2\right\} 
\end{equation*}%
which is an ellipsoid. From inequality \ref{desigualdad} we obtain that 
\begin{equation*}
A\subset \mathfrak{E}_{1}\cap \mathfrak{E}_{2}\subset \mathfrak{E}%
_{3}\subset \mathfrak{E}_{4}.
\end{equation*}%
Let $\mathbf{z}\in \mathbb{K}^{n}$ be the vector with coordinates $%
c_{i}/\left( \lambda _{i}+1\right) $. We have 
\begin{equation*}
\mathfrak{E}_{4}=\mathbf{z}+\left\{ x\in \mathbb{K}^{n}\mid \sum \left(
\lambda _{i}+1\right) \left\vert x_{i}\right\vert ^{2}\leq 2\right\} =%
\mathfrak{El}\left( \left( \frac{\mathbf{\lambda }+1}{2}\right) ^{-%
{\frac12}%
}\right) +\mathbf{z}.
\end{equation*}%
Since $\det \mathbf{\beta }=1$, we obtain $\det \mathbf{\lambda }=1.$ If $%
\mathbf{\lambda }\neq \mathbf{1}$the hypothesis of Lemma \ref{volumen} are
satisfied and then%
\begin{equation*}
\Delta \overset{\text{def}}{=}\det \left( \frac{\mathbf{\lambda }+1}{2}%
\right) >1.
\end{equation*}%
Therefore, $\func{Vol}\mathfrak{E}_{4}=\Delta ^{-%
{\frac12}%
}\func{Vol}\mathfrak{E}_{2}<\func{Vol}\mathfrak{E}_{2}$ which contradicts
the minimality of $\mathfrak{E}_{2}$.

So we conclude that $\mathbf{\lambda }=\mathbf{1}$ , $\mathfrak{E}_{1}$ is
the unit ball and that $\mathfrak{E}_{2}=\mathfrak{E}_{1}+\mathbf{c}$. In
this case the inequality \ref{desigualdad} transforms to the following%
\begin{equation*}
\sum \left\vert \left( x_{i}-\frac{c_{i}}{2}\right) \right\vert ^{2}\leq
1-\sum \left\vert \frac{c_{i}}{2}\right\vert ^{2}.
\end{equation*}%
Therefore $\mathfrak{E}_{3}$ is a ball with center in $\mathbf{c}/2$, and
radius 
\begin{equation*}
r\overset{\text{def}}{=}\sqrt{1-\sum \left\vert \frac{c_{i}}{2}\right\vert
^{2}}.
\end{equation*}%
If $\mathbf{c}$ is not the origin then $r<1$ and has volume strictly less
than that of the unit ball $\mathfrak{E}_{1}$. Therefore $\mathbf{c}=\mathbf{%
0}$ and $\mathfrak{E}_{2}=\mathfrak{E}_{1}.$
\end{proof}

\section{Applications}

\subsection{Centered ellipsoids}

There are other extremal ellipsoids when the extremum is searched for
ellipsoids that have a fixed center (say the origin). In this setting all
extremal ellipsoids are unique, maximal inscribed or minimal circumscribed;
real or complex.

To see this, one can modify the proofs of the previous theorems (which is
not difficult) or use the symmetrization of sets. If $A$ is a set in $%
\mathbb{K}^{n}$, its symmetrization around the origin is%
\begin{equation*}
\mathcal{S}\left( A\right) =\bigcup_{\left\vert \zeta \right\vert =1}\zeta A
\end{equation*}%
It is not hard to prove that if $A$ is non-flat and compact then $\mathcal{S}%
\left( A\right) $ also is non-flat and compact.

Applying previous theorems we obtain unique extremal ellipsoids
circumscribed or inscribed in $\mathcal{S}\left( A\right) .$ These extremal
ellipsoids contain the origin and it is not difficult to see that they are
the unique extremal ellipsoids centered at the origin.

\subsection{Polarity}

As one is introduced in the subject of L\"{o}wner--John Ellipsoids there is
a strong feeling of certain symmetry in the concepts and proofs. This
subsection partially explains that feeling.

Let $A$ be a non-flat compact convex subset of $\mathbb{R}^{n}$ \textit{%
containing the origin in its interior}. The polar of $A$ is%
\begin{equation*}
A^{\ast }=\left\{ \mathbf{x}\in \mathbb{R}^{n}\mid \mathbf{x}\cdot \mathbf{a}%
\leq 1\quad \forall \mathbf{a}\leq 1\right\} .
\end{equation*}%
Polarity reverses the inclusion relation. It is a fact that $A^{\ast \ast }=A
$. (see for example \cite{Barvinok} Chapter IV). If $B$ is a symmetric
definite positive matrix then 
\begin{equation*}
\left\{ B\mathbf{u}\mid \mathbf{u}\in \mathfrak{B}\left( \mathbb{R}%
^{n}\right) \right\} ^{\ast }=\left\{ B^{-1}\mathbf{u}\mid \mathbf{u}\in 
\mathfrak{B}\left( \mathbb{R}^{n}\right) \right\} 
\end{equation*}%
i.e. the polars of ellipsoids centered at the origin are ellipsoids centered
at the origin. Moreover, if $\mathfrak{E}$ is an ellipsoid centered at the
origin then 
\begin{equation}
\func{Vol}\left( \mathfrak{E}\right) \func{Vol}\left( \mathfrak{E}^{\ast
}\right) =1.  \label{polar volume}
\end{equation}

Denote by $\mathcal{MI}\left( A\right) $, the ellipsoid of minimal volume
among all ellipsoids containing $A$ and centered at the origin. Denote by $%
\mathcal{MA}\left( A\right) $ ,the ellipsoid of maximal volume among all
ellipsoids contained in $A$ and centered at the origin.

\begin{theorem}
$\mathcal{MA}\left( A\right) ^{\ast }=\mathcal{MI}\left( A^{\ast }\right) $.
\end{theorem}

\begin{proof}
Denote $E=\mathcal{MA}\left( A\right) $ and $F=\mathcal{MI}\left( A^{\ast
}\right) $. We have $A^{\ast }\subset E^{\ast }$ and by the minimality of $F$
we have $\func{Vol}F\leq \func{Vol}E^{\ast }$. By equation \ref{polar volume}
$\func{Vol}F\func{Vol}E\leq 1$. On the other hand, we have $F^{\ast }\subset
A$ and by the maximality of $E$ we have $\func{Vol}E\geq \func{Vol}F^{\ast }$%
. By equation \ref{polar volume}, $\func{Vol}F\func{Vol}E\geq 1$. Therefore, 
$\func{Vol}F\func{Vol}E=1$.

The ellipsoid $E^{\ast }$ contains $A^{\ast }$ and has the same volume as $F$%
. By the uniqueness of $\mathcal{MI}\left( A^{\ast }\right) $ we have $%
E^{\ast }=F$.
\end{proof}

\begin{corollary}
The two theorems about the uniqueness of extremal real ellipsoids centered
at the origin are polar of each other.
\end{corollary}

\begin{remark}
There is no polarity in $\mathbb{C}^{n}$.
\end{remark}

\subsection{A characterization of ellipsoids}

\begin{theorem}
\label{caracterizacion transitiva}Let $A$ be a non-flat compact subset of $%
\mathbb{K}^{n}$. Then $A$ is an ellipsoid if given any two non-interior
points of $A$, there is an affine isomorphism which maps one of them to the
other and preserves $A.$
\end{theorem}

\begin{proof}
Let $f$ be an affine isomorphism that preserves $A$ then, it must preserve
the convex closure of $A$. Since $A$ is non-flat, its convex closure has
non-cero volume. This implies that $f$ preserves volumes. Let $\mathfrak{E}$
be the unique ellipsoid of minimum volume containing $A$. The ellipsoid $%
f\left( \mathfrak{E}\right) $ contains $f\left( A\right) =A$ and has the
same volume than $\mathfrak{E}.$ Since $\mathfrak{E}$ is unique $f\left( 
\mathfrak{E}\right) =\mathfrak{E}$. Since any affinity is continuous $%
f\left( \partial \mathfrak{E}\right) =\partial \mathfrak{E}$ (the boundary
of $\mathfrak{E}$).

Denote by $\partial A$ the set of non interior points of $A.$ Since $%
\mathfrak{E}$ is minimal, there exists $\mathbf{p}\in \partial \mathfrak{E}%
\cap \partial A.$ By hypothesis, for any point $\mathbf{q}\in \partial A$
there is an affinity $f\ $which preserves $A$ and $f\left( \mathbf{p}\right)
=\mathbf{q}$. Since $f\left( \partial \mathfrak{E}\right) =\partial 
\mathfrak{E,}$ then $\mathbf{q}\in \partial \mathfrak{E}.$ This means that $%
\partial A\subseteq \partial \mathfrak{E}$. Reciprocally, let $\mathbf{r}$
be any point in $\partial \mathfrak{E}$. If $\mathbf{r}$ is an interior
point of $A,$ then there is a small ball with center in $\mathbf{r}$
contained in $A$ but not in $\mathfrak{E\ }$which is not possible because $%
A\subset \mathfrak{E}$. This means that $\partial A=\partial \mathfrak{E}$
and proves that $A$ is convex. Therefore $A=\widehat{\partial A}=\widehat{%
\partial \mathfrak{E}}=\mathfrak{E}$.
\end{proof}

\subsection{Brunn's Theorem}

Let $A$ be a non-flat compact subset of $\mathbb{K}^{n}$. We will call $A$ a 
\emph{puck} if any non empty intersection with a line is convex and
symmetric i.e. if $\mathbb{K=R}$ then it is a line segment; if $\mathbb{K=C}$
then it is a disk. When $\mathbb{K=R}$ pucks are just convex bodies.

\begin{theorem}
A puck $A$ in $\mathbb{K}^{n}$ is an ellipsoid if and only if for any line $%
\ell $ the centers of all intersections of $A$ with lines parallel to $\ell $%
, lie in a hyperplane.
\end{theorem}

\begin{proof}
For the only if part, observe that the property holds for balls and it's
preserved by affine transformations. We shall prove the if part using
Theorem \ref{caracterizacion transitiva}. For this, let $\mathbf{x}$ and $%
\mathbf{y}$ be two different non-interior points\ of $A$ an let $\ell $ be
the only affine line which contains $\mathbf{x}$ and $\mathbf{y}$. The line $%
\ell $ defines an hyperplane $H$ as stated in the hypothesis. Let $\mathbf{o}
$ be the point $\ell \cap H$. Let us make a translation $t_{\mathbf{o}}$
which sends $\mathbf{o}$ to the origin. Denote $A^{\prime }=A-\mathbf{o}$, $%
\mathbf{x}^{\prime }=\mathbf{x}-\mathbf{o}$, $\mathbf{y}^{\prime }=\mathbf{y}%
-\mathbf{o,}$ etc.

Since $\mathbf{x}^{\prime }$ and $\mathbf{y}^{\prime }$ belong to the same
linear subspace $\ell ^{\prime }$ which is of dimension 1, hence there
exists an scalar $\lambda \in \mathbb{K}$ such that $\mathbf{y}^{\prime
}=\lambda \mathbf{x}^{\prime }$. The center of $\ell \cap A$ is $\mathbf{o}$%
, therefore the center of $\ell ^{\prime }\cap A^{\prime }$ is the origin.
Moreover since $\mathbf{x}^{\prime }$ and $\mathbf{y}^{\prime }$ are
non-interior points of $A^{\prime }$ they must lie in the boundary of $\ell
^{\prime }\cap A^{\prime }$ and therefore $\left\Vert \mathbf{x}^{\prime
}\right\Vert =\left\Vert \mathbf{y}^{\prime }\right\Vert .$ Hence we obtain $%
\left\vert \lambda \right\vert =\left\Vert \mathbf{y}^{\prime }\right\Vert
\left\Vert \mathbf{x}^{\prime }\right\Vert ^{-1}=1$.

The linear subspaces $H^{\prime }$ and $\ell ^{\prime }$ are complementary.
Therefore, for any vector $\mathbf{z}\in \mathbb{K}^{n}$ there exists unique
vectors $\mathbf{z}_{\ell ^{\prime }}\in \ell ^{\prime }$ and $\mathbf{z}%
_{H^{\prime }}\in H^{\prime }$ such that $\mathbf{z}=\mathbf{z}_{\mathbf{%
\ell ^{\prime }}}+\mathbf{z}_{H^{\prime }}$. Define $\phi \left( \mathbf{z}%
\right) =\lambda \mathbf{z}_{\ell ^{\prime }}+\mathbf{z}_{H^{\prime }}%
\mathbf{\ }$which is a linear isomorphism of the whole space. Observe that $%
\phi \left( \mathbf{x}^{\prime }\right) =\lambda \mathbf{x}^{\prime }=%
\mathbf{y}^{\prime }$. The spaces $\ell ^{\prime }$ and $H^{\prime }$ are
invariant subspaces of $\phi $; and $\phi $ is the direct sum of the
identity in $H^{\prime }$ and the multiplication by $\lambda $ in $\ell
^{\prime }$. Any parallel line $L$ to $\ell ^{\prime }$ is equal to $\ell
^{\prime }+\mathbf{p}$ where $\mathbf{p}=L\cap H^{\prime }$. If $\mathbf{z}%
\in L$ then $\phi \left( \mathbf{z}\right) =\lambda \mathbf{z}_{\ell
^{\prime }}+\mathbf{p}$. Therefore $\phi $ leaves invariant $L,$ and it acts
inside $L$ like multiplication by $\lambda .$ Since $L\cap A^{\prime }$ is
symmetric, $\left\vert \lambda \right\vert =1$ and the center of $L\cap
A^{\prime }$ lies in $H^{\prime }$ then $\phi \left( L\cap A^{\prime
}\right) =L\cap A^{\prime }.$ We have that $A^{\prime }$ is the disjoint
union of all the $L\cap A^{\prime }$ with $L$ parallel to $\ell ^{\prime }.$
Since $\phi $ preserves each \textquotedblleft slice\textquotedblright\ then 
$\phi $ preserves $A^{\prime }$.

Now consider the affine isomorphism $f:\mathbb{K}^{n}\ni \mathbf{z}\mapsto
t_{\mathbf{o}}^{-1}\left( \phi \left( t_{\mathbf{o}}\left( \mathbf{z}\right)
\right) \right) \in \mathbb{K}^{n}$. We have that $f\left( A\right) =t_{%
\mathbf{o}}^{-1}\left( \phi \left( A^{\prime }\right) \right) =A$ and $%
f\left( \mathbf{x}\right) =t_{\mathbf{o}}^{-1}\left( \phi \left( \mathbf{x}%
^{\prime }\right) \right) =t_{\mathbf{o}}^{-1}\left( \mathbf{y}^{\prime
}\right) =\mathbf{y}$. So the hypothesis of theorem \ref{caracterizacion
transitiva} is fulfilled and therefore $A$ is an ellipsoid.
\end{proof}

\begin{remark}
If $\mathbb{K=R}$ then the affine isomorphism $f$ from the previous proof is
usually called skew reflection. If $\mathbb{K=C}$ then $f$ has no settled
down name and the authors preferred to describe it than to name it.
\end{remark}

\section{Auxiliary facts}

This section contains two lemmas needed in the paper whose proofs are
straightforward computations. Recall that $\mathbf{1}$ is the vector which
has all coordinates equal to $1$.

\begin{lemma}
\label{volumen}If $\mathbf{\lambda }\in \mathbb{R}_{+}^{n}$ , $\det \mathbf{%
\lambda }=1$ and $\mathbf{\lambda }\neq \mathbf{1}$then%
\begin{equation*}
\det \left( \frac{\mathbf{\lambda }+\mathbf{1}}{2}\right) >1\text{.}
\end{equation*}
\end{lemma}

\begin{proof}
We have $\left( \lambda _{i}-1\right) ^{2}\geq 0$ and therefore $\left(
\lambda _{i}+1\right) ^{2}\geq 4\lambda _{i}$.With equality only when $%
\lambda _{i}=1.$ Taking the product, we get $\det \left( \mathbf{\lambda }+%
\mathbf{1}\right) ^{2}>4^{n}\det \mathbf{\lambda }$ which is the same as $%
\det \left( \mathbf{\lambda }+\mathbf{1}\right) >2^{n}\sqrt{\det \mathbf{%
\lambda }}=2^{n}.$ This proves the lemma.
\end{proof}

\begin{lemma}
\label{Lema para el afin}Let $c,x$ be complex numbers and $\lambda $ a real
number. Then%
\begin{equation*}
\lambda \left\vert x\right\vert ^{2}+\left\vert x-c\right\vert ^{2}=\left(
\lambda +1\right) \left\vert x-\frac{c}{\left( \lambda +1\right) }%
\right\vert ^{2}+\left( \frac{\lambda }{\lambda +1}\right) \left\vert
c\right\vert ^{2}
\end{equation*}
\end{lemma}

\begin{proof}
We have%
\begin{equation*}
\begin{array}{c}
\lambda \left\vert x\right\vert ^{2}+\left\vert x-c\right\vert ^{2}=\lambda
\left\vert x\right\vert ^{2}+\left( x-c\right) \overline{\left( x-c\right) }
\\ 
=\left( \lambda +1\right) \left\vert x\right\vert ^{2}-\left( x\overline{c}+c%
\overline{x}\right) +\left\vert c\right\vert ^{2}%
\end{array}%
\end{equation*}%
and completing the square we obtain%
\begin{equation*}
\begin{array}{l}
=\left( \lambda +1\right) \left( \left\vert x\right\vert ^{2}-\dfrac{\left( x%
\overline{c}+c\overline{x}\right) }{\left( \lambda +1\right) }+\left\vert 
\dfrac{c}{\left( \lambda +1\right) }\right\vert ^{2}-\left\vert \dfrac{c}{%
\left( \lambda +1\right) }\right\vert ^{2}\right) +\left\vert c\right\vert
^{2} \\ 
=\left( \lambda +1\right) \left\vert x-\dfrac{c}{\left( \lambda +1\right) }%
\right\vert ^{2}+\left( \dfrac{\lambda }{\lambda +1}\right) \left\vert
c\right\vert ^{2}.%
\end{array}%
\end{equation*}
\end{proof}

\end{document}